\documentclass[11pt,letterpaper,reqno]{amsart}
\usepackage{tikz}
\usetikzlibrary{positioning, shapes.geometric, arrows.meta, calc, positioning}
\usepackage{amssymb}
\usepackage{amsmath}
\usepackage{amsthm}
\usepackage{amsfonts}
\usepackage{bbm}
\usepackage{enumitem} 
\usepackage{pgfplots}
\pgfplotsset{compat=1.18} 
\usepackage{booktabs}
\usepackage{graphicx}
\usepackage[T1]{fontenc}
\usepackage{doi}
\usepackage{float} 
\usepackage{bookmark}
\usepackage{hyperref}
\hypersetup{pdfstartview={FitH}}

\newcommand{\abs}[1]{| #1 |}

\newtheorem{thm}{Theorem}[section]
\newtheorem{lem}[thm]{Lemma}
\newtheorem{prop}[thm]{Proposition}
\newtheorem{cor}[thm]{Corollary}
\newtheorem{claim}[thm]{Claim}

\newtheorem{prob}[thm]{Problem}

\theoremstyle{definition}

\newtheorem{defn}[thm]{Definition}
\numberwithin{equation}{section}

\newcommand{\lcm}{\operatorname{lcm}}
\makeatother

\begin{document}


\title{Harmonic LCM patterns and sunflower-free capacity}

\author[Q.~Tang]{Quanyu Tang}
\author[S.~Zhang]{Shengtong Zhang}

\address{School of Mathematics and Statistics, Xi'an Jiaotong University, Xi'an 710049, P. R. China}
\email{tang\_quanyu@163.com}

\address{Department of Mathematics, Stanford University, Stanford, CA 94305, USA}
\email{stzh1555@stanford.edu}



\begin{abstract}
Fix an integer $k\ge 3$. Call a set $A\subseteq [N]$ \emph{LCM-$k$-free} if it does not contain
distinct $a_1,\dots,a_k$ such that $\mathrm{lcm}(a_i,a_j)$ is the same for all $1\le i<j\le k$.
Define
$$
f_k(N):=\max\left\{\sum_{a\in A}\frac1a: A\subseteq [N] \text{ is LCM-$k$-free}\right\}.
$$
Addressing a problem of Erd\H{o}s, we prove an explicit unconditional lower bound
$$
f_k(N)\ge (\log N)^{c_k-o(1)},
\qquad
c_k:=\frac{k-2}{e((k-2)!)^{1/(k-2)}}.
$$
Let $F_k(n)$ denote the maximum size of a $k$-sunflower-free family of subsets of $[n]$, and define the Erd\H{o}s--Szemer\'edi $k$-sunflower-free capacity by $\mu_k^{\mathrm S}:=\limsup_{n\to\infty}F_k(n)^{1/n}$. Motivated by a remark of Erd\H{o}s relating this problem to the sunflower conjecture, we show that
$$
(\log N)^{\log\mu_k^{\mathrm S}-o(1)}  \le f_k(N) \ll (\log N)^{\mu_k^{\mathrm S}-1+o(1)}.
$$
Furthermore, we show that the Erd\H{o}s--Szemer\'edi sunflower conjecture fails for this fixed $k$
(i.e. $\mu_k^{\mathrm S}=2$) if and only if $f_k(N)=(\log N)^{1-o(1)}$.
\end{abstract}


\maketitle


\section{Introduction}

\subsection{The problem and background}Fix integers $k\ge 3$ and $N\ge 1$. Following Erd\H{o}s~\cite{Er70}, we study subsets
$A\subseteq [N]:=\{1,2,\dots,N\}$ that avoid certain ``equal--LCM'' configurations.
We say that $A$ \emph{contains an LCM-$k$-tuple} if there exist distinct
$a_1,\dots,a_k\in A$ such that all pairwise least common multiples coincide, i.e.
\[
\lcm(a_i,a_j)=\lcm(a_1,a_2)\qquad (1\le i<j\le k).
\]
We say $A$ is \emph{LCM-$k$-free} if it contains no LCM-$k$-tuple.

In~\cite[p.~124]{Er70}, Erd\H{o}s introduced the counting function $g(k,x)$, defined as
the largest integer $s$ for which one can find
\[
1\le a_1<\cdots<a_s\le x
\]
with the property that no $k$ of the $a_i$ form an LCM-$k$-tuple. Erd\H{o}s
originally conjectured that $g(k,x)=o(x)$ for every $k\ge 3$, but later proved that
this fails for all $k\ge 4$; see~\cite[Theorem~1]{Er70}. In particular, for each
$k\ge 4$ there exists $\delta_k>0$ such that
\[
g(k,x)\ge \delta_k x
\]
for all sufficiently large $x$.

Erd\H{o}s also considered a harmonic analogue. He showed~\cite[Eq.~(19)]{Er70} that for each
$k\ge 3$ there exists $b_k>0$ such that whenever $x$ is sufficiently large and
$1\le a_1<\cdots<a_s\le x$ satisfy
\begin{equation}\label{eq:Erdos-threshold}
\sum_{i=1}^s \frac{1}{a_i} > b_k \log x,
\end{equation}
then some $k$ of the $a_i$ form an LCM-$k$-tuple. In~\cite[p.~127]{Er70}, Erd\H{o}s wrote:
\begin{quote}
\emph{I do not know how much \eqref{eq:Erdos-threshold} can be weakened so that there should always
be $k$ $a$'s every two of which have the same least common multiple.}
\end{quote}
If we define
\[
f_k(N)\;:=\;\max\left\{\sum_{a\in A}\frac1a:\ A\subseteq [N]\ \text{is LCM-$k$-free}\right\}.
\]
then Erd\H{o}s is asking the following question.
\begin{prob}\label{prob:856}
Let $k\ge 3$. Estimate $f_k(N)$.
\end{prob}

Problem~\ref{prob:856} appears as Problem~\#856 on Bloom's Erd\H{o}s Problems website~\cite{EP856}
and may be viewed as the harmonic counterpart of the counting problem introduced by Erd\H{o}s
in~\cite[p.~124]{Er70}.

A trivial bound is
\[
f_k(N)\le \sum_{n\le N}\frac1n = \log N + O(1).
\]

Erd\H{o}s~\cite{Er70} provided a short proof of the improved upper bound $f_k(N)\ll_k \log N/\log\log N$. It is observed in the comment section on the Erd\H{o}s Problems website~\cite{EP856}
that one can further improve this to $f_k(N)\leq \log N \cdot \exp\left(-\Omega_k\left(\frac{\log\log N}{\log\log\log N}\right)\right)$. To our knowledge, no lower bounds of comparable strength were previously available.

\subsection{A polylogarithmic lower bound}Our first result provides an explicit polylogarithmic lower bound.

\begin{thm}\label{thm:lower}
Fix an integer $k\ge 3$. Then, as $N\to\infty$,
\[
f_k(N)\ge (\log N)^{c_k-o(1)},
\qquad
c_k:=\frac{k-2}{e ((k-2)!)^{1/(k-2)}}.
\]
\end{thm}

Our construction uses a ``prime bucketing'' scheme, which will be used for all lower bound constructions. We partition primes in a suitable interval
into many disjoint blocks of comparable harmonic sum and form squarefree integers by choosing
exactly $k-2$ primes from each block. A short combinatorial lemma forces LCM-$k$-freeness, while the
harmonic sum factorizes and can be optimized over a free parameter.

\subsection{Connections with the sunflower problem}
We next investigate the optimal exponent in the polylogarithmic bound.  Erd\H{o}s suggested that Problem~\ref{prob:856} is related to the sunflower problem. Recall that a collection of $k$ distinct sets $S_1,\dots,S_k$ is called a \emph{$k$-sunflower} if
\[
S_i\cap S_j = S_1\cap S_2 \qquad \text{for all }1\le i<j\le k.
\]
This common intersection is called the \emph{kernel}. A family of sets $\mathcal F$ is said to be
\emph{$k$-sunflower-free} if it contains no $k$ distinct members forming a $k$-sunflower. Dually, a collection of $k$ distinct sets $S_1,\dots,S_k$ is called a \emph{$k$-cosunflower} if \[
S_i\cup S_j = S_1\cup S_2 \qquad \text{for all }1\le i<j\le k.
\] A family $\mathcal F$ is \emph{$k$-cosunflower-free} if it
contains no $k$ distinct members forming a $k$-cosunflower.


In~\cite[p.~127]{Er70}, Erd\H{o}s wrote:
\begin{quote}
\emph{Problem~\ref{prob:856} seems connected with the following combinatorial problem:}
\end{quote}

\begin{prob}\label{prob:857}
Let $m=m(n,k)$ be minimal such that any distinct collection of sets
$A_1,\dots,A_m\subseteq [n]$ contains a $k$-sunflower. Estimate $m(n,k)$, or ideally determine
its asymptotic behavior.
\end{prob}

Problem~\ref{prob:857} is widely known as the Erd\H{o}s--Szemer\'edi sunflower conjecture~\cite{ErSz78},
and it appears as Problem~\#857 on Bloom's Erd\H{o}s Problems website~\cite{EP857}. Following the
notation of~\cite{NaSa17}, let $F_k(n)$ be the maximum size of a $k$-sunflower-free family
$\mathcal{F}\subseteq 2^{[n]}$, and define the \emph{Erd\H{o}s--Szemer\'edi $k$-sunflower-free capacity}
\[
\mu_k^{\mathrm S}:=\limsup_{n\to\infty} F_k(n)^{1/n}.
\]Trivially $\mu_k^{\mathrm S}\le 2$, and the celebrated Erd\H{o}s--Szemer\'edi sunflower conjecture~\cite{ErSz78} asserts that
$\mu_k^{\mathrm S}<2$ for every $k\ge 3$. A standard tensor power argument\footnote{Here is a concise description of the argument. Let $\mathcal{F}_n\subseteq 2^{[n]}$ be a $k$-sunflower-free family. Choose $r\in\{0,1,\dots,n\}$ maximizing $|\mathcal{F}_n\cap\binom{[n]}{r}|$, and set $\mathcal{F}_n':=\mathcal{F}_n\cap\binom{[n]}{r}$, so that $|\mathcal{F}_n'|\ge \frac{1}{n+1}|\mathcal{F}_n|$. For each integer $t\ge 2$, partition $[tn]$ into $t$ subsets $S_1\sqcup\cdots\sqcup S_t$ with $|S_i|=n$. For each $i$ fix a bijection $\pi_i:[n]\to S_i$ and define a copy of $\mathcal{F}_n'$ on $S_i$ by $\mathcal{F}_n'^{(i)}:=\{\pi_i(F):F\in\mathcal{F}_n'\}\subseteq 2^{S_i}$. Then one checks that $\mathcal{F}_{n}^{\otimes t}:=\{F_1\sqcup\cdots\sqcup F_t:\ F_i\in\mathcal{F}_n'^{(i)}\}\subseteq 2^{[tn]}$ is $k$-sunflower-free and satisfies $|\mathcal{F}_{n}^{\otimes t}|=|\mathcal{F}_n'|^t\ge (|\mathcal{F}_n|/(n+1))^t$.} shows that the limsup can be replaced with a limit:
\begin{equation}
    \label{eq:lim}
  \mu_k^{\mathrm S} =\lim_{n\to\infty} F_k(n)^{1/n}.  
\end{equation}


In this paper, we make Erd\H{o}s's suggested connection between Problem~\ref{prob:856} and Problem~\ref{prob:857} precise by showing the Erd\H{o}s--Szemer\'edi sunflower conjecture is equivalent to a tight bound on $f_k(N)$.

\begin{thm}\label{thm:equivalence}
For each $k\ge 3$, we have $\mu_k^{\mathrm S} = 2$ if and only if $f_k(N) = (\log N)^{1 - o(1)}$.
\end{thm}
The implication ``$\Rightarrow$'' is obtained by constructing a LCM-$k$-free set of squarefree integers
whose prime-support sets form a large cosunflower-free family in an appropriate weighted
product measure. The implication ``$\Leftarrow$'' follows from Theorem~\ref{thm:B-implies-A} below. 
\begin{thm}\label{thm:B-implies-A}
Fix $k\ge 3$. Then, as $N\to\infty$,
\[
f_k(N)\ll (\log N)^{\mu_k^{\mathrm S}-1+o(1)}.
\]
\end{thm}
Its proof is a double-counting argument similar to Erd\H{o}s's original argument in~\cite{Er70}. To each integer $m$ we
associate a set system on the prime divisors of $m$ and show that LCM-$k$-freeness forces these set
systems to be $k$-sunflower-free. We then combine the resulting estimate with a Sathe--Selberg type
lower bound for the harmonic sum of squarefree almost primes. 

Finally, we wonder how good a lower bound for $f_k(N)$ we can establish given knowledge of $\mu_k^{\mathrm S}$. While we can obtain a lower bound by following the proof of Theorem~\ref{thm:equivalence}, it is not good when $\mu_k^{\mathrm S}$ is smaller than $2$. Towards an improvement on Theorem~\ref{thm:lower}, we show that
\begin{thm}\label{thm:mu-lb-implies-f-lb}
Fix $k\ge 3$. Then, as $N\to\infty$,
\[f_k(N) \ge (\log N)^{\log\mu_k^{\mathrm S} - o(1)}.\]
\end{thm}
For $k=3$, the best current upper bound due to Naslund and Sawin~\cite[Theorem~1]{NaSa17} states that
\[
\mu_3^{\mathrm S}\le \frac{3}{2^{2/3}}=1.889881574\dots\, .
\]
In the opposite direction, a construction of Deuber, Erd\H{o}s, Gunderson, Kostochka and
Meyer~\cite{DEGKM97} yields \(\mu_3^{\mathrm S}>1.551\). Hence, by Theorems~\ref{thm:B-implies-A} and~\ref{thm:mu-lb-implies-f-lb}, we obtain the following consequence.
\begin{cor}\label{cor:f3-upper}
As $N\to\infty$,
\[
(\log N)^{\log(1.551) - o(1)} \leq f_3(N)\ll (\log N)^{\frac{3}{2^{2/3}}-1+o(1)}.
\]
\end{cor}
Unfortunately, Theorem~\ref{thm:mu-lb-implies-f-lb} never gives an exponent greater than $\log 2 < 0.7$, so it is strictly worse than Theorem~\ref{thm:lower} for $k \geq 7$. We leave it as an open question to obtain the correct exponent.
\begin{prob}
    In terms of $\mu_k^{\mathrm S}$, determine the smallest $c > 0$ such that for all sufficiently large $N$, we have
    $$f_k(N) \leq (\log N)^{c + o(1)}.$$
\end{prob}

The rest of the paper is organized as follows.
In Section~\ref{sec:preliminaries} we fix notation and record two analytic inputs.
In Section~\ref{sec:lower} we prove our unconditional lower bound
(Theorem~\ref{thm:lower}). In Section~\ref{sec:upper} we prove the capacity-based upper bound (Theorem~\ref{thm:B-implies-A}).
In Section~\ref{sec:equivalence} we prove the full-density equivalence in
Theorem~\ref{thm:equivalence}. In Section~\ref{sec:lower-bound-via-capacity} we establish the capacity-based lower bound Theorem~\ref{thm:mu-lb-implies-f-lb}.
Finally, the appendix contains the proof of the technical Lemma~\ref{lem:harmonic-SS}.

\section*{Acknowledgements}
The authors would like to express their sincere gratitude to Thomas Bloom and Terence Tao
for valuable comments posted on the Erd\H{o}s Problems website, which inspired the present work.
They are also grateful to Thomas Bloom for founding and maintaining the Erd\H{o}s Problems website.

\section{Preliminaries}\label{sec:preliminaries}

\subsection{Notation and asymptotic conventions}

Throughout the paper, $\log$ denotes the natural logarithm. For a positive integer \(n\), let \(\omega(n)\) denote the number of distinct prime factors of \(n\),
and let \(\Omega(n)\) denote the number of prime factors of \(n\) counted with multiplicity. For a condition $E$ (e.g.\ $i\in F$), we write $\mathbf{1}_{E}$ for the indicator of $E$,
namely $\mathbf{1}_{E}=1$ if $E$ holds and $\mathbf{1}_{E}=0$ otherwise.

We use Vinogradov's asymptotic notation. For functions \(f=f(n)\) and \(g=g(n)\), we write
\(f=O(g)\), \(g=\Omega(f)\), \(f\ll g\), or \(g\gg f\) to mean that there exists a constant \(C>0\) such that
\(|f(n)|\le C g(n)\) for all sufficiently large \(n\).
We write \(f\asymp g\) or \(f=\Theta(g)\) to mean that \(f\ll g\) and \(g\ll f\), and we write
\(f=o(g)\) to mean that \(f(n)/g(n)\to 0\) as \(n\to\infty\).
If the implicit constant is allowed to depend on one or more parameters $z_1,\dots,z_r$, we indicate this by writing $f\ll_{z_1,\dots,z_r} g$, $g\gg_{z_1,\dots,z_r} f$, $f=O_{z_1,\dots,z_r}(g)$ or $g=\Omega_{z_1,\dots,z_r}(f)$.

\subsection{Auxiliary analytic estimates}
In this subsection we record two standard analytic inputs that will be used in the proof of
Theorem~\ref{thm:B-implies-A}.

\begin{lem}\label{lem:lemma22}
Let $z>0$ be a fixed real number. Then for all $X\ge 2$ we have
\[
\sum_{m\le X}\frac{z^{\omega(m)}}{m}
\ll_z
(\log X)^z.
\]
\end{lem}

\begin{proof}
Fix $X\ge 2$ and a real number $z>0$. Since every integer $m\le X$ has all its prime divisors $\le X$, we have the crude but useful inequality
\[
\sum_{m\le X}\frac{z^{\omega(m)}}{m}
\le
\sum_{\substack{m\ge 1\\ \text{all prime divisors of }m\text{ are }\le X}}
\frac{z^{\omega(m)}}{m}
=: S(X).
\]Because the arithmetic function $n\mapsto z^{\omega(n)}$ is multiplicative, we can expand $S(X)$ as a finite Euler product over primes $p\le X$:
\[
S(X)
= \prod_{p\le X}\left(\sum_{k=0}^\infty\frac{z^{\omega(p^k)}}{p^k}\right)= \prod_{p\le X}\left(1 + \sum_{k \ge 1} \frac{z}{p^k}\right)= \prod_{p\le X}\left(1 + \frac{z}{p-1}\right).        
\]Taking logarithms and using $\log(1+u) \le u$,
\[
\log S(X)
= \sum_{p\le X} \log \left(1 + \frac{z}{p-1}\right) \leq z\sum_{p\le X} \frac{1}{p-1}= z\log\log X + O(z),
\]where we used the standard estimate $\sum_{p\le X} (p-1)^{-1} = \log\log X + O(1)$. Thus
$S(X) \ll_z (\log X)^z$, and the claim follows.
\end{proof}

\begin{lem}\label{lem:harmonic-SS}
Let
\[
H_\ell(N) := \sum_{\substack{n\le N\\ \omega(n)=\Omega(n)=\ell}} \frac1n.
\]Then for every $0<\eta<1$ we have
\begin{equation}\label{eq:H-ell-lower}
  H_\ell(N) \gg_\eta \frac{(\log\log N)^\ell}{\ell!}
  \qquad (N\to\infty,\ 1\le \ell\le (1-\eta)\log\log N).
\end{equation}
\end{lem}

\begin{proof}
See Appendix~\ref{appendix:SSS}.
\end{proof}

\section{An unconditional lower bound}
\label{sec:lower}

We start from a simple combinatorial lemma.

\begin{lem}\label{lem:r=k-2-union}
Let $k\ge 3$. If $S_1,\dots,S_k$ are $(k - 2)$-element sets with the same pairwise union, then $S_1=\cdots=S_k$.
\end{lem}
\begin{proof}
Let $U = \cup_{j = 1}^k S_j$. Write $u:=|U|$ and $r:=k-2$. For each $j$ set $M_j:=U\setminus S_j$.
Then $M_a\cap M_b=\varnothing$ for $a\ne b$, since otherwise an element missing
from both $S_a$ and $S_b$ would be missing from $S_a\cup S_b=U$.
Also $|M_j|=u-r$ for all $j$ because $|S_j|=r$.
Thus
\[
k(u-r)=\sum_{j=1}^k |M_j| = \left|\bigsqcup_{j=1}^k M_j\right|\le |U|=u,
\]
so $u\le \frac{kr}{k-1}=\frac{k(k-2)}{k-1}<k-1=r+1$. Hence $u\le r$.
But $u\ge r$ since each $S_j$ has size $r$ and lies in $U$. Therefore $u=r$,
so $u-r=0$ and every $M_j$ is empty. Thus $S_j=U$ for all $j$.
\end{proof}

\begin{proof}[Proof of Theorem~\ref{thm:lower}]
Write $L:=\log\log N$ and fix $k\ge 3$, $r:=k-2$.  Let
\[
\delta:=\delta(N):=L^{-1/2},
\qquad
t:=\left\lfloor \frac{(1-\delta)L}{B}\right\rfloor,
\qquad
x:=N^{1/(rt)},
\qquad
y:=\exp(L^{1/3}),
\]
where $B>1$ is a constant to be chosen later.  For $N$ large we have $t\ge 1$,
$y<x$, and $1/p\le 1/y\le \delta$ for every prime $p\ge y$.

By Mertens' theorem for primes, we have
\begin{equation}\label{eq:enough-total-sum}
\sum_{y<p\le x}\frac1p
=\log\log x-\log\log y+O(1) = L - \log(rt)-\frac{1}{3}\log L +O(1) \ge t(B+\delta)
\end{equation}for $N$ sufficiently large. 

List the primes in $[y,x]$ in increasing order, and form disjoint sets
$P_1,P_2,\dots$ greedily as follows: start with $P_1=\varnothing$ and keep
adding the next unused prime $p$ to the current set until its harmonic sum
exceeds $B$, then start a new set. Because each prime $p\in[y,x]$ satisfies $1/p\le \delta$, each completed set
$P_i$ satisfies
\begin{equation}\label{eq:bucket-sum-11}
B \le \sum_{p\in P_i}\frac1p < B+\delta.
\end{equation}
So \eqref{eq:enough-total-sum} implies that this process produces at
least $t$ completed sets $P_1,\dots,P_t$.

Define
\[
A_N
:=\left\{
n=\prod_{i=1}^t\ \prod_{p\in S_i} p
\ :\ S_i\subseteq P_i,\ |S_i|=r\ \text{for all }i
\right\}.
\]
Every $n\in A_N$ is squarefree and satisfies $n\le x^{rt}=N$, hence
$A_N\subseteq\{1,2,\dots,N\}$.

We claim that $A_N$ is LCM-$k$-free. Suppose for contradiction that
$n_1,\dots,n_k\in A_N$ is an LCM-$k$-tuple. Write
\[
n_j=\prod_{i=1}^t \Bigl(\prod_{p\in S_i^{(j)}}p\Bigr),
\qquad
S_i^{(j)}\subseteq P_i,\ \ |S_i^{(j)}|=r.
\]
Since the prime sets $P_i$ are disjoint across different $i$, equality of the
pairwise $\lcm$'s implies that for each fixed $i$ the unions
$S_i^{(a)}\cup S_i^{(b)}$ are independent of the pair $(a,b)$. Let $U_i$ denote
this common union. Applying Lemma~\ref{lem:r=k-2-union} in each bucket $i$ yields
$S_i^{(1)}=\cdots=S_i^{(k)}$ for every $1\le i\le t$, hence $n_1=\cdots=n_k$,
contradicting that the $n_j$ are distinct. Therefore $A_N$ is LCM-$k$-free.

For each $i$ define
\[
E_i:=\sum_{\substack{S\subseteq P_i\\ |S|=r}}\ \frac{1}{\prod_{p\in S}p}.
\]
Then by multiplicativity across buckets, we have \(\sum_{n\in A_N}\frac1n=\prod_{i=1}^t E_i\). Write $J_i:=\sum_{p\in P_i}1/p$, so by
\eqref{eq:bucket-sum-11} we have $B\le J_i<B+\delta$ and $\max_{p\in P_i}1/p\le 1/y\le\delta$.
Expanding $J_i^r$ we have
\[
J_i^r = r! E_i + R_i,
\]
where $R_i$ is the contribution of $r$-tuples with at least one repeated index.
Choosing a repeated pair of positions and bounding the remaining factors by $J_i$
gives the crude but sufficient estimate
\[
R_i
\le
\sum_{1\le u<v\le r}\ \sum_{\substack{(p_1,\dots,p_r)\in P_i^r\\ p_u=p_v}}
\frac{1}{p_1\cdots p_r} \le \binom{r}{2}\left(\max_{p\in P_i}\frac{1}{p}\right)J_i^{r-1}
 \le \binom{r}{2}\delta(B+\delta)^{r-1}.
\]
Thus
\[
E_i \ge \frac{J_i^r}{r!}-\frac{\binom{r}{2}}{r!} \delta(B+\delta)^{r-1}  \ge \frac{B^r}{r!}\left(1-O_{r}(\delta)\right),
\]
uniformly in $1\le i\le t$ (here $r=k-2$ is fixed). Therefore we have
\[
\sum_{n\in A_N}\frac1n
=\prod_{i=1}^t E_i
\ \ge\
\left(\frac{B^r}{r!}\left(1-O_{r}(\delta)\right)\right)^t.
\]
Since $\delta=L^{-1/2}$ and $t=(1+o(1))L/B$, where $L=\log\log N$, it follows that
\[
\left(1-O_r(\delta)\right)^t=\exp \left(t \log\left(1-O_{r}(\delta)\right)\right)=\exp \left(-O_r(t\delta)\right)
=\exp (-O_r(L^{\frac{1}{2}}))=(\log N)^{-o(1)}.
\]Moreover,
\[
\left(\frac{B^r}{r!}\right)^t
=\exp \left(t(r\log B-\log(r!))\right)
=(\log N)^{\frac{r\log B-\log(r!)}{B}-o(1)}.
\]Combining these estimates yields
\[
\sum_{n\in A_N}\frac1n  \ge (\log N)^{\frac{r\log B-\log(r!)}{B}-o(1)}.
\]
Since $A_N$ is LCM-$k$-free, we conclude that
\[
f_k(N) \ge (\log N)^{\frac{r\log B-\log(r!)}{B}-o(1)}.
\]
Recall that $r = k -2$. Taking
$B=e (r!)^{1/r}$ gives
\[
f_k(N) \ge (\log N)^{r/(e(r!)^{1/r})-o(1)}=(\log N)^{c_k-o(1)}.
\]
This completes the proof.\end{proof}

\section{An upper bound via sunflower-free capacity}\label{sec:upper}

We now prove Theorem~\ref{thm:B-implies-A} by encoding LCM-$k$-free sets as sunflower-free families and applying a harmonic estimate for almost primes.

\begin{proof}[Proof of Theorem~\ref{thm:B-implies-A}]
Fix \(k\ge 3\), write \(\mu:=\mu_k^{\mathrm S}\) and let \(F_k(n)\) denote the size of the largest \(k\)-sunflower-free family $\mathcal{F}\subseteq 2^{[n]}$. Then for every $\delta>0$ there exists a constant $C_1=C_1(\delta)>0$ such that
\begin{equation}\label{eq:F3-capacity}
F_k(n)\le C_1 (\mu+\delta)^n\qquad\text{for all }n\ge 1.
\end{equation}

Fix a set $A\subseteq\{1,\dots,N\}$ which is LCM-$k$-free and write $\mathcal{P}(m)$ for the set of prime divisors of $m$. For integers $\ell,m \ge 1$, let $r_\ell(m)$ be the number of representations
\begin{equation}\label{eq:r-def}
m=a\prod_{p\in S} p,
\end{equation}
where $a\in A$ and $S\subseteq \mathcal P(m)$ is an $\ell$-element set of primes.
Equivalently, define
\[
\mathcal F_m:=\left\{S\subseteq \mathcal P(m): |S|=\ell,\ m=a\prod_{p\in S}p\ \text{for some }a\in A\right\}.
\]
Note that for fixed $m$ and $S$, the value of $a$ is uniquely determined, hence $r_\ell(m)=|\mathcal F_m|$.


\begin{claim}\label{claim112221}
For every $m$ the family $\mathcal{F}_m$ is $k$-sunflower-free.
\end{claim}

\begin{proof}[Proof of Claim~\ref{claim112221}]
Suppose, for a contradiction, that $\mathcal{F}_m$ contains a $k$-sunflower
$S_1,\dots,S_k$ with common kernel $K$ and pairwise disjoint petals
$P_i := S_i\setminus K$. For each $i$, there exists distinct $a_i\in A$ such that
\[
m = a_i\prod_{p\in S_i} p.
\]
Since $S_i\cap S_j = K$ for all $1\le i<j\le k$, we deduce
\[
\operatorname{lcm}(a_i,a_j)
= \frac{m}{\gcd(m / a_i, m / a_j)}
= \frac{m}{\prod_{p\in K} p}
\quad\text{for all }1\le i<j\le k.
\]
Thus the $k$ elements $a_1,\dots,a_k\in A$  are distinct and have the same pairwise lcm. This contradicts the assumption that $A$ is LCM-$k$-free. 
\end{proof}

The family $\mathcal{F}_m$ lives on the ground set $\mathcal{P}(m)$ of size
$\omega(m)$. Therefore, by Claim~\ref{claim112221} and the bound
\eqref{eq:F3-capacity}, for any fixed $\delta>0$ we have
\begin{equation}\label{eq:r-ell-m-bound}
  r_\ell(m)
  = |\mathcal{F}_m|
  \le F_k\bigl(\omega(m)\bigr)
  \le C_1(\delta) \bigl(\mu+\delta\bigr)^{\omega(m)}.
\end{equation}

Fix $\ell\ge 1$ and set
\[
H_\ell(N) := \sum_{\substack{n\le N\\\omega(n)=\Omega(n)=\ell}}\frac1n.
\]
Consider the double sum
\[
T := \biggl(\sum_{a\in A}\frac1a\biggr) H_\ell(N) = \sum_{a\in A}\frac1a
     \sum_{\substack{n\le N\\\omega(n)=\Omega(n)=\ell}}\frac1n.
\]
Put $m:=an$. Then $m\le N^2$. Each such pair $(a,n)$ produces a representation~\eqref{eq:r-def} of $m$
(with the prime factors of $n$ as the distinct primes $p_1,\dots,p_\ell$), hence contributes to $r_\ell(m)$.
Consequently,
\[
T = \sum_{a\in A}\sum_{\substack{n\le N\\\omega(n)=\Omega(n)=\ell}}\frac1{an}
  \leq \sum_{m\le N^2}\frac{r_\ell(m)}{m}.
\]
Using \eqref{eq:r-ell-m-bound}, we get
\[
T  \le C_1(\delta)\sum_{m\le N^2}
\frac{(\mu+\delta)^{\omega(m)}}{m}.
\]Taking $z=\mu+\delta$ and $X=N^2$ in Lemma~\ref{lem:lemma22}, we obtain
\[
 C_1(\delta)\sum_{m\le N^2}
\frac{(\mu+\delta)^{\omega(m)}}{m} \ll_{\delta} (\log N+\log 2)^{\mu+\delta}\ll_{\delta} (\log N)^{\mu+\delta}.
\]
This means
\begin{equation}\label{eq:master-ineq}
 \biggl(\sum_{a\in A}\frac1a\biggr) H_\ell(N) \ll_{\delta} (\log N)^{\mu+\delta}.
\end{equation}Let $L := \log\log N$. Let $\eta\in(0,1)$ be fixed for the moment, and restrict to integers $\ell$ with $1\le \ell\le (1-\eta)L$.
By Lemma~\ref{lem:harmonic-SS}, $H_\ell(N)\gg_\eta L^\ell/\ell!$, so \eqref{eq:master-ineq} implies
\begin{equation}\label{eq:sumA-basic}
\sum_{a\in A}\frac1a \ll_{\eta,\delta} (\log N)^{\mu+\delta} \frac{\ell!}{L^\ell}.
\end{equation}

We now choose $\ell$ as a function of $N$ so as to optimise the decay of the
factor $\ell!/L^\ell$. Choose $\ell:=\lfloor(1-\eta)L\rfloor$. By Stirling's formula,
\[
\ell! = \sqrt{2 \pi \ell}\left(\frac{\ell}{e}\right)^\ell \left(1+O\left(\frac{1}{\ell}\right)\right),
\]and hence
\begin{equation}\label{eq:ell-over-L1}
\frac{\ell!}{L^\ell}
= \exp \Bigl(\ell(\log\ell -1 -\log L)+O(\log\ell)\Bigr).
\end{equation}
Write $\lambda:=\ell/L$. Then $\lambda=1-\eta+O(1/L)$ and
$\log\ell = \log L + \log\lambda$. Therefore
\begin{equation}\label{eq:ell-over-L2}
\ell(\log\ell -1 -\log L)
= \lambda L(\log\lambda -1)
= (1-\eta)L\bigl(\log(1-\eta)-1\bigr) + o(L).
\end{equation}Plugging \eqref{eq:ell-over-L1} and \eqref{eq:ell-over-L2} into \eqref{eq:sumA-basic} gives
\[
\sum_{a\in A}\frac1a
 \ll_{\eta,\delta} 
(\log N)^{\mu+\delta}
\exp\Bigl( (1-\eta)\left(\log(1-\eta)-1\right)L + o(L)\Bigr).
\]
Recalling that $L=\log\log N$, this may be rewritten as
\begin{equation}\label{eq:ell-over-L6}
\sum_{a\in A}\frac1a
 \ll_{\eta,\delta}
(\log N)^{\Psi(\eta,\delta)+o(1)},
\end{equation}
where \(\Psi(\eta,\delta) := \mu+\delta + (1-\eta)\bigl(\log(1-\eta)-1\bigr)\).

We now analyse $\Psi(\eta,\delta)$. For fixed $\delta>0$, the function $\eta\mapsto\Psi(\eta,\delta)$ is continuous
on $[0,1)$, and
\[
\lim_{\eta\to 0^+}\Psi(\eta,\delta) = \mu+\delta-1.
\]Now let $\varepsilon>0$ be arbitrary. First choose $\delta =\varepsilon/2$. By continuity of $\Psi(\cdot,\delta)$ at $\eta=0$,
there exists $\eta_0=\eta_0(\varepsilon)>0$ such that for all
$0<\eta\le\eta_0$ we have
\[
\Psi(\eta,\delta) \le \mu+\delta-1 + \varepsilon/2.
\]
For such $\eta$, \(\Psi(\eta,\delta)
\le \mu -1 + \delta + \varepsilon/2
\le \mu -1 + \varepsilon\). Recalling~\eqref{eq:ell-over-L6}, we conclude that for every $\varepsilon>0$, we have
\[
\sum_{a\in A}\frac1a
  \ll_\varepsilon(\log N)^{\mu-1+\varepsilon}.
\]
Since $A\subseteq\{1,\dots,N\}$ was an arbitrary LCM-$k$-free set, it follows that
\[
f_k(N) \ll_\varepsilon (\log N)^{\mu-1+\varepsilon}.
\]This completes the proof.
\end{proof}

\section{Equivalence at full density}\label{sec:equivalence}

In this section we prove Theorem~\ref{thm:equivalence}.
The direction $f_k(N)=(\log N)^{1-o(1)} \Rightarrow \mu_k^{\mathrm S} = 2$ is immediate from
Theorem~\ref{thm:B-implies-A} by comparing exponents.
The reverse direction is more involved. 


Assume $\mu_k^{\mathrm S} = 2$. We first obtain large uniform
$k$-cosunflower-free families at an arbitrary fixed density via a tensor/padding argument of
Alon--Shpilka--Umans~\cite{AAC12}.
We then upgrade this to a weighted statement under a product measure using a different ``blow-up construction'',
and finally encode the resulting cosunflower-free family by squarefree integers to produce a
LCM-$k$-free set with near-maximal harmonic sum.

\subsection{Blow-ups and product measures}

Within a fixed ground set, taking complements swaps $k$-sunflowers and $k$-cosunflowers:
for $S_1,\dots,S_k\subseteq U$ we have
\[
S_i\cup S_j\ \text{is constant for all }i<j
\quad\Longleftrightarrow\quad
(U\setminus S_i)\cap(U\setminus S_j)\ \text{is constant for all }i<j.
\]
We will work with cosunflowers. The following property is easy to check.
\begin{prop}\label{prop:pairwise-union-characterization}
Let $k\ge 3$ and let $S_1,\dots,S_k$ be sets.
The pairwise unions $S_i\cup S_j$ are independent of the pair $1\le i<j\le k$
if and only if every element belongs to either $0$ or at least $k-1$ of the sets $S_1,\dots,S_k$. 
\end{prop}
\begin{defn}[Blow-up]\label{defn:blowup}
Let $U$ be a finite set equipped with a partition
\[
U=U_1\sqcup\cdots\sqcup U_t\sqcup \widetilde{U}.
\]
For $F\subseteq [t]$, define the \emph{blow-up of $F$ over $(\{U_i\}_{i=1}^t,\widetilde U)$} to be the family
\[
B(F,\{U_i\}_{i=1}^t,\widetilde{U})
:=\left\{P\subseteq U:\ |P\cap U_i|=\mathbf{1}_{i\in F}\ \text{for all }i\in[t],\ \text{and }P\cap \widetilde{U}=\varnothing \right\}.
\]
For a family $\mathcal F\subseteq 2^{[t]}$, define its blow-up over $(\{U_i\}_{i=1}^t,\widetilde U)$ by
\[
B(\mathcal F,\{U_i\}_{i=1}^t,\widetilde{U})
:=\bigsqcup_{F\in\mathcal F} B(F,\{U_i\}_{i=1}^t,\widetilde{U}),
\]
where the union is disjoint by definition.
\end{defn}

\begin{prop}\label{prop:blowup-ok}
Let $k\ge 3$. In the notation of Definition~\ref{defn:blowup}, if $\mathcal F\subseteq 2^{[t]}$
is $k$-cosunflower-free, then its blow-up
$B(\mathcal F,\{U_i\}_{i=1}^t,\widetilde U)\subseteq 2^{U}$
is also $k$-cosunflower-free.
\end{prop}

\begin{proof}
Suppose for a contradiction that $B(\mathcal F,\{U_i\}_{i=1}^t,\widetilde U)$ contains a
$k$-cosunflower $P_1,\dots,P_k$, i.e.\ the sets are distinct and their pairwise unions coincide.
For each $j$ define the index set
\[
F_j:=\{i\in[t]:\ P_j\cap U_i\neq\varnothing\}\in\mathcal F.
\]

Fix $i\in[t]$. Since each $P_j$ meets $U_i$ in either $0$ or $1$ element, an element of $U_i$ belongs
to exactly the number of those $P_j$ that choose it.
By Proposition~\ref{prop:pairwise-union-characterization} applied to $P_1,\dots,P_k$,
no element may belong to between $1$ and $k-2$ of the $P_j$.
Therefore the number of indices $j$ with $i\in F_j$ is either $0$ or at least $k-1$.
Applying Proposition~\ref{prop:pairwise-union-characterization} to the sets $F_1,\dots,F_k\subseteq[t]$
shows that their pairwise unions are constant.

If $F_u=F_v$ for some $u\neq v$, then $P_u$ and $P_v$ select the same set of blocks.
Since $P_u\neq P_v$, there is a block $U_i$ with $i\in F_u$ such that
$P_u\cap U_i=\{a\}$ and $P_v\cap U_i=\{b\}$ for distinct $a\neq b$.
Both $a$ and $b$ belong to at least one of the sets $P_1,\dots,P_k$, so by
Proposition~\ref{prop:pairwise-union-characterization} each must belong to at least $k-1$ of them.
This is impossible because $2(k-1)>k$ for $k\ge 3$.
Hence the sets $F_1,\dots,F_k$ are distinct.

Therefore, $F_1,\dots,F_k$ are distinct members of $\mathcal F$ with constant pairwise union, i.e.\ a
$k$-cosunflower in $\mathcal F$, contradicting that $\mathcal F$ is $k$-cosunflower-free.
\end{proof}

\begin{defn}[Product measure]\label{defn:prob_measure_1}
Let $A$ be a finite set, and associate weights $w_a\in[0,1]$ to each $a\in A$.
Define the \emph{product measure} $\mu_w$ on $2^A$ by
\[
\mu_w(B) := \prod_{a\in B} w_a \prod_{a\in A\setminus B} (1-w_a),
\qquad B\subseteq A.
\]
For a family $\mathcal A\subseteq 2^A$, write $\mu_w(\mathcal A):=\sum_{B\in\mathcal A}\mu_w(B)$. Equivalently, $\mu_w(B)$ is the probability that a random subset of $A$, formed by including each
$a\in A$ independently with probability $w_a$, is exactly $B$.
\end{defn}

\subsection{Uniform cosunflower-free families at fixed density}

We next recall a variation of Problem~\ref{prob:857} studied by Alon, Shpilka and Umans~\cite{AAC12},
and record a convenient ``for all large $N$'' consequence that will be used in the weighted argument.

\begin{thm}\label{thm:ASU}
Fix an integer \(k\ge 3\). For a rational \(\beta >1\), let \(H(\beta,k)\) denote the following hypothesis:
\begin{quote}
\(H(\beta,k)\): There exists \(\eta>0\) such that for every integer \(n\) with \(\beta n\in\mathbb Z\), every family
\(\mathcal F\subseteq \binom{[\beta n]}{n}\) with
\(|\mathcal F|\ge \binom{\beta n}{n}^{\,1-\eta}\)
contains a \(k\)-sunflower.
\end{quote}
Then for any rationals \(\beta,\beta'>1\), the statements \(H(\beta,k)\) and \(H(\beta',k)\) are equivalent.

In particular, if \(\mu_k^{\mathrm S} = 2\), then \(H(\beta,k)\) fails for every rational \(\beta>1\),
and thus for every fixed rational \(\alpha\in(0,1)\) and every \(\eta>0\), there exist a constant \(\widetilde{N}=\widetilde{N}(\alpha,\eta)\) such that for every integer \(N\ge \widetilde{N}\) there is an
\(\lfloor \alpha N\rfloor\)-uniform, \(k\)-cosunflower-free family
\[
\mathcal G_N\subseteq \binom{[N]}{\lfloor \alpha N\rfloor}
\qquad\text{with}\qquad
|\mathcal G_N| \ge \binom{N}{\lfloor \alpha N\rfloor}^{1-\eta}.
\]
\end{thm}
\begin{proof}
The equivalence of \(H(\beta,k)\) for different rationals \(\beta>1\), as well as the stated consequence from
\(\mu_k^{\mathrm S} = 2\), is proved in \cite[Theorem~2.4]{AAC12} (for \(k=3\)); the same argument applies to any
fixed \(k\ge 3\) since it uses only heredity of being sunflower-free under restrictions and Cartesian products. In addition, their argument, together with \eqref{eq:lim}, easily gives the ``all sufficiently large $N$'' part of the theorem.
\end{proof}

\subsection{A weighted cosunflower-free family of large product measure}

We now combine Theorem~\ref{thm:ASU} with the blow-up construction to obtain a new weighted
version of Theorem~\ref{thm:ASU}.

\begin{lem}\label{lem:weighted-sunflower}
Suppose \(\mu_k^{\mathrm S} = 2\). Then for any \(\epsilon>0\), there exists \(K=K(\epsilon)>0\)
such that the following holds. Let \(A\) be a finite set, and associate weights
\(w_a\in[0,1]\) to each \(a\in A\). Assume:
\begin{enumerate}
\item The total weight \(W:=\sum_{a\in A}w_a\) satisfies \(W\ge K\).
\item Each weight satisfies \(w_a\le K^{-1}\).
\end{enumerate}
Then there exists a \(k\)-cosunflower-free family \(\mathcal A\subseteq 2^A\) such that
\[
\mu_w(\mathcal A) \ge 2^{-\epsilon W}.
\]
\end{lem}

\begin{proof}
Fix \(\epsilon>0\). Since the conclusion for \(\epsilon_0\) implies the conclusion for any
\(\epsilon\ge \epsilon_0\), we may assume \(0<\epsilon\le 1/10\).

\medskip
\noindent\textbf{1. Partition.}
Let \(c:=1/\lceil 1/\epsilon^2\rceil\). Then \(c\in(0,1)\cap\mathbb Q\) and \(c\asymp \epsilon^2\);
in particular \(c\le \epsilon^2\le \epsilon/10\).
We iteratively construct a partition
\[
A=A_1\sqcup A_2\sqcup\cdots\sqcup A_n\sqcup \widetilde A
\]
as follows: as long as there exists a subset of the remaining elements with total weight at least \(c\),
take \(A_i\) to be a minimal such subset and remove it; at the end, put the remaining elements into
\(\widetilde A\). By minimality,
\begin{equation}\label{eq:minimal_condition_c}
\sum_{a\in A_i}w_a \in [c, c+K^{-1}],
\end{equation}
and \(\sum_{b\in \widetilde A}w_b\le c\).
Consequently,
\begin{equation}\label{eq:partition_construction_wck_1}
n \in \left[\frac{W-c}{c+K^{-1}}, \frac{W}{c}\right].
\end{equation}
Define
\[
r_i:=\prod_{a\in A_i}(1-w_a),\qquad
q_i:=\sum_{a\in A_i}w_a\prod_{b\in A_i\setminus\{a\}}(1-w_b),\qquad
s:=\prod_{b\in\widetilde A}(1-w_b).
\]
Thus \(r_i\) is the probability that a \(\mu_w\)-random subset chooses no elements from \(A_i\),
and \(q_i\) is the probability it chooses exactly one element from \(A_i\).

\medskip
\noindent\textbf{2. A large uniform index family.}
Set
\[
\delta:=\frac{\epsilon c}{10}.
\]
Apply Theorem~\ref{thm:ASU} with \(\alpha=c\) and \(\eta=\delta\). By~\eqref{eq:partition_construction_wck_1} we have
\(n \ge \frac{W-c}{c+K^{-1}}\), so by increasing \(K\) (depending on \(\epsilon\)) we may assume that
\(W\ge K \ge 100\) and \(n\ge \widetilde{N}(c,\delta)\). Hence Theorem~\ref{thm:ASU}
provides a \(k\)-cosunflower-free family
\[
\mathcal F\subseteq \binom{[n]}{m},\qquad m:=\lfloor cn\rfloor,
\]
satisfying
\begin{equation}\label{eq:F-size}
|\mathcal F| \ge \binom{n}{m}^{1-\delta}.
\end{equation}

\medskip
\noindent\textbf{3. Blow-up.}
Let
\[
\mathcal A:=B(\mathcal F,\{A_i\}_{i=1}^n,\widetilde{A})
\]
be the blow-up of $\mathcal F$ over $(\{A_i\}_{i=1}^n,\widetilde{A})$. By Proposition~\ref{prop:blowup-ok}, \(\mathcal A\) is \(k\)-cosunflower-free.

\medskip
\noindent\textbf{4. Lower bounding \(\mu_w(\mathcal A)\).}
For each \(F\in\mathcal F\), by independence across the blocks we have
\[
\mu_w\left(B(F,\{A_i\}_{i=1}^n,\widetilde{A})\right)= s\prod_{i\in F} q_i\prod_{i\notin F} r_i.
\]
We now lower bound \(r_i,q_i,s\). Since \(K\ge 100\), we know that \(w_a\le K^{-1} \le 0.01\) and
\(
1-x\ge e^{-x-x^2}
\)
holds for all \(x\in[0,0.01]\).
Then
\[
r_i=\prod_{a\in A_i}(1-w_a) \ge \exp\left(-\sum_{a\in A_i}(w_a+w_a^2)\right).
\]
Since \(w_a^2\le K^{-1}w_a\) and \eqref{eq:minimal_condition_c} holds, we get
\[
\sum_{a\in A_i}(w_a+w_a^2)\le (1+K^{-1})\sum_{a\in A_i}w_a
\le (1+K^{-1})(c+K^{-1})\le c+2 K^{-1}.
\]
Hence
\begin{equation}\label{eq:ri-lb}
r_i \ge \exp\left(-c-2 K^{-1}\right).
\end{equation}
Moreover,
\[
\frac{q_i}{r_i}
=\sum_{a\in A_i}\frac{w_a}{1-w_a}
\ge \sum_{a\in A_i}w_a
\ge c,
\]
so
\begin{equation}\label{eq:qi-lb}
q_i \ge c\exp\left(-c-2 K^{-1}\right).
\end{equation}
Similarly, since \(\sum_{b\in \widetilde A}w_b\le c\), we have
\begin{equation}\label{eq:s-lb}
s \ge \exp\left(-c-2 K^{-1}\right).
\end{equation}
Since every \(F\in\mathcal F\) has size \(|F|=m\), combining \eqref{eq:ri-lb}, \eqref{eq:qi-lb}, \eqref{eq:s-lb} yields
\[
\mu_w\left(B(F,\{A_i\}_{i=1}^n,\widetilde{A})\right)
= s\prod_{i\in F} q_i\prod_{i\notin F} r_i
\ge c^{m}\exp\Bigl(-\left(c+2 K^{-1}\right)(n+1)\Bigr).
\]
Summing over \(F\in\mathcal F\) and using \eqref{eq:F-size} gives
\begin{equation}\label{eq:mu-start}
\mu_w(\mathcal A)
\ge \binom{n}{m}^{1-\delta} c^{m}
\exp\left(-\left(c+2 K^{-1}\right)(n+1)\right).
\end{equation}

\medskip
\noindent\textbf{5. Estimating the RHS.}
Since \(\binom{n}{m}\le 2^n\le e^n\), we have
\[
\binom{n}{m}^{1-\delta}
=\binom{n}{m}\cdot \binom{n}{m}^{-\delta}
\ge \binom{n}{m}\,e^{-\delta n}.
\]
Substituting this into \eqref{eq:mu-start} and choosing \(K\) so large that
\(
\frac{2}{K}(n+1)\le \delta n
\)
(for all \(n\) in our range; it suffices to take \(K\ge 4/\delta\)),
we obtain
\begin{equation}\label{eq:mu-mid}
\mu_w(\mathcal A)
\ge \binom{n}{m}\,c^{m}\,e^{-c(n+1)}\,e^{-2\delta n}.
\end{equation}
Next we use the elementary lower bound (valid for \(0<k<n\)):
\begin{equation}\label{eq:binom-lb}
\binom{n}{k} \ge \frac{1}{n+1}
\left(\frac{n}{k}\right)^{k}\left(\frac{n}{n-k}\right)^{n-k}.
\end{equation}
Applying \eqref{eq:binom-lb} with \(k=m\) and multiplying by \(c^m\), using \(m\le cn\), gives
\[
\binom{n}{m}c^m
\ge \frac{1}{n+1}\left(\frac{cn}{m}\right)^m\left(\frac{n}{n-m}\right)^{n-m}
\ge \frac{1}{n+1}\left(\frac{n}{n-m}\right)^{n-m}.
\]
Since \(m=\lfloor cn\rfloor\ge cn-1\), we have
\(
n-m\le (1-c)n+1
\),
and hence \(\frac{n}{n-m} \ge \frac{1}{1-c+1/n}\). Therefore,
\[
\binom{n}{m}c^m
\ge \frac{1}{n+1}\left(1-c+1/n\right)^{-(n-m)}
\ge \frac{1}{n+1}\exp\Bigl((c-1/n)(n-m)\Bigr),
\]
where we used \(1-x\le e^{-x}\).
Substituting this into \eqref{eq:mu-mid} yields
\begin{align*}
\mu_w(\mathcal A)
&\ge \frac{1}{n+1}
\exp\Bigl((c-1/n)(n-m)\Bigr)\exp\Bigl(-c(n+1)\Bigr)\exp(-2\delta n)\\
&= \exp\Bigl(-c(m+1)-\frac{n-m}{n}\Bigr)\exp(-2\delta n)\\
&\ge \exp\bigl(-c(m+1)-1\bigr)\exp(-2\delta n),
\end{align*}
since \((n-m)/n\le 1\). Now \(m\le cn\) implies \(c(m+1)\le c^2n+c\), and by~\eqref{eq:partition_construction_wck_1} we have \(n\le W/c\) and thus \(c^2n\le cW\). By the definition of \(\delta\), we know that
\[
2\delta n \le 2\cdot \frac{\epsilon c}{10}\cdot \frac{W}{c}=\frac{\epsilon}{5}W.
\]
Thus
\[
\mu_w(\mathcal A)
\ge \exp\bigl(-cW-c-1\bigr)\exp\Bigl(-\frac{\epsilon}{5}W\Bigr).
\]
Finally, let \(\lambda=\log 2-\frac{1}{5}>0\). Since \(c\le \epsilon^2 \le \epsilon/10\), for \(W\ge 101/\epsilon\) we have \(cW+c+1\le \lambda \epsilon W\), hence
\[
\exp(-cW-c-1) \ge \exp(-\lambda \epsilon W).
\]Putting everything together,
\[
\mu_w(\mathcal A)
\ge \exp(-\lambda \epsilon W)\exp\Bigl(-\frac{\epsilon}{5}W\Bigr)
=\exp(-(\log 2)\epsilon W)
=2^{-\epsilon W},
\]
as desired.
\end{proof}

\subsection{Proof of Theorem~\ref{thm:equivalence}}

\begin{proof}[Proof of Theorem~\ref{thm:equivalence}]
\emph{``If'' direction.}
Assume that $f_k(N)=(\log N)^{1-o(1)}$ as $N\to\infty$.
By Theorem~\ref{thm:B-implies-A} we have
\[
f_k(N)\ll_k (\log N)^{\mu_k^{\mathrm S}-1+o(1)}.
\]
Comparing exponents forces $\mu_k^{\mathrm S}-1\ge 1$, hence $\mu_k^{\mathrm S}\ge 2$.
Since trivially $\mu_k^{\mathrm S}\le 2$, we conclude that $\mu_k^{\mathrm S}=2$.

\medskip\noindent
\emph{``Only if'' direction.}
Assume that $\mu_k^{\mathrm S} = 2$.
Fix $\epsilon>0$. We will construct, for all sufficiently large $N$, a LCM-$k$-free set $A\subseteq[N]$
with
\[
\sum_{a\in A}\frac{1}{a}\ \gg_{\epsilon}\ (\log N)^{1-\epsilon}.
\]
Since trivially $f_k(N)\le \sum_{n\le N}\frac1n=\log N+O(1)$, this implies
$f_k(N)=(\log N)^{1-o(1)}$.

Let $K:=K(\epsilon/2)$ be as in Lemma~\ref{lem:weighted-sunflower} (applied with parameter $\epsilon/2$).
Let $T=T(N)\to\infty$ be a threshold to be chosen later, and let $P$ denote the set of primes in $[K,T]$.
For each $p\in P$ set
\[
w_p:=\frac{1}{p+1}.
\]
Then $w_p\le (K+1)^{-1}\le K^{-1}$ for all $p\in P$, and by Mertens' theorem
\[
W:=\sum_{p\in P} w_p=\sum_{p\in P}\frac{1}{p+1}=\log\log T+O_{\epsilon}(1),
\]
so for $T$ sufficiently large we have $W\ge K$.
Applying Lemma~\ref{lem:weighted-sunflower} (with $\epsilon/2$) to the ground set $P$ with weights $(w_p)_{p\in P}$,
we obtain a $k$-cosunflower-free family $\mathcal Q\subseteq 2^{P}$ such that
\[
\mu_w(\mathcal Q) \ge 2^{-\epsilon W/2}.
\]Define
\[
A_0:=\left\{ \prod_{p\in Q}p:\ Q\in\mathcal Q\right\},
\qquad
A:=A_0\cap [N].
\]

\begin{claim}\label{claim:the_final_clain_11-1}
$A_0$ (and hence $A$) is LCM-$k$-free.    
\end{claim}
\begin{proof}[Proof of Claim~\ref{claim:the_final_clain_11-1}]
Each element of $A_0$ is squarefree and has all prime factors in $P$.
If $a=\prod_{p\in Q}p$ and $b=\prod_{p\in R}p$ with $Q,R\subseteq P$, then
\[
\operatorname{lcm}(a,b)=\prod_{p\in Q\cup R}p.
\]
Thus, if $a_1,\dots,a_k\in A_0$ are distinct and satisfy $\operatorname{lcm}(a_i,a_j)$ constant for all $i<j$,
then their prime support sets $Q_i\subseteq P$ satisfy $Q_i\cup Q_j$ constant for all $i<j$.
Hence $Q_1,\dots,Q_k$ form a $k$-cosunflower in $\mathcal Q$, contradicting that $\mathcal Q$ is
$k$-cosunflower-free. This proves the claim.
\end{proof}

For $Q\subseteq P$ we have
\[
\mu_w(Q)=\prod_{p\in Q}w_p\prod_{p\in P\setminus Q}(1-w_p),
\qquad
\mu_w(\emptyset)=\prod_{p\in P}(1-w_p),
\]
and since $w_p/(1-w_p)=1/p$, it follows that
\[
\frac{\mu_w(Q)}{\mu_w(\emptyset)}=\prod_{p\in Q}\frac{w_p}{1-w_p}=\frac{1}{\prod_{p\in Q}p}.
\]
Therefore,
\[
\sum_{a\in A_0}\frac{1}{a}
=\sum_{Q\in\mathcal Q}\frac{1}{\prod_{p\in Q}p}
=\sum_{Q\in\mathcal Q}\frac{\mu_w(Q)}{\mu_w(\emptyset)}
=\frac{\mu_w(\mathcal Q)}{\mu_w(\emptyset)}
\ge 2^{-\epsilon W/2}\cdot \frac{1}{\mu_w(\emptyset)}.
\]
Moreover,
\[
\frac{1}{\mu_w(\emptyset)}=\prod_{p\in P}\frac{1}{1-w_p}=\prod_{p\in P}\left(1+\frac{1}{p}\right)
\gg_{\epsilon} \log T,
\]
again by Mertens' theorem (the implied constant depends on $K$, hence on $\epsilon$).
Combining this with $W=\log\log T+O_{\epsilon}(1)$ yields
\begin{equation}\label{eq:A0-harmonic-lb}
\sum_{a\in A_0}\frac{1}{a}
\gg_{\epsilon} (\log T)^{1-\frac{\log 2}{2}\epsilon}.
\end{equation}It remains to bound the contribution of those $a\in A_0$ with $a>N$.
Note that if $\prod_{p\in Q}p\ge N$, then $\sum_{p\in Q}\log p\ge \log N$, hence for $a\in A_0\setminus [N]$ we have
\[
\frac{1}{\prod_{p\in Q}p}
\le \frac{1}{\log N}\cdot \frac{\sum_{p\in Q}\log p}{\prod_{p\in Q}p}.
\]
Therefore,
\begin{align*}
\sum_{a\in A_0\setminus [N]}\frac{1}{a}
&\le \sum_{\substack{Q\subseteq P\\ \prod_{p\in Q}p\ge N}}\frac{1}{\prod_{p\in Q}p}
\le \frac{1}{\log N}\sum_{Q\subseteq P}\frac{\sum_{p\in Q}\log p}{\prod_{p\in Q}p}\\
&= \frac{1}{\log N}\sum_{p\in P}\frac{\log p}{p}\prod_{q\in P\setminus\{p\}}\left(1+\frac{1}{q}\right)\\
&\le \frac{1}{\log N}\left(\prod_{q\in P}\left(1+\tfrac{1}{q}\right)\right)\sum_{p\in P}\frac{\log p}{p}.
\end{align*}
Using $\prod_{q\in P}(1+\frac{1}{q})\ll_{\epsilon}\log T$ and $\sum_{p\in P}\frac{\log p}{p}=\log T+O_{\epsilon}(1)$,
we obtain
\begin{equation}\label{eq:tail-bound}
\sum_{a\in A_0\setminus [N]}\frac{1}{a}
= O_{\epsilon}\left(\frac{\log^2 T}{\log N}\right).
\end{equation}

Combining \eqref{eq:A0-harmonic-lb} and \eqref{eq:tail-bound} gives
\[
\sum_{a\in A}\frac{1}{a}
\ge
\sum_{a\in A_0}\frac{1}{a}-\sum_{a\in A_0\setminus [N]}\frac{1}{a}
\gg_{\epsilon} (\log T)^{1-\frac{\log 2}{2}\epsilon}
-O_{\epsilon}\Bigl(\frac{\log^2 T}{\log N}\Bigr).
\]Finally choose $T=T(N)$ so that
\[
\log T = (\log N)^{1-\frac{3}{5}\epsilon}.
\]
Then the main term satisfies
\[
(\log T)^{1-\frac{\log 2}{2}\epsilon}
=(\log N)^{(1-\frac{3}{5}\epsilon)(1-\frac{\log 2}{2}\epsilon)}
\gg (\log N)^{1-\epsilon},
\]
while the error term satisfies
\[
\frac{\log^2 T}{\log N}=(\log N)^{1-\frac{6}{5}\epsilon}
=o\bigl((\log N)^{1-\epsilon}\bigr).
\]
Hence for all sufficiently large $N$ we have $\sum_{a\in A}\frac1a\gg_\epsilon (\log N)^{1-\epsilon}$. Since $\epsilon>0$ was arbitrary and, by Claim~\ref{claim:the_final_clain_11-1}, the set $A$ is LCM-$k$-free,
we conclude that $f_k(N)=(\log N)^{1-o(1)}$.
\end{proof}

\section{A lower bound via sunflower-free capacity}
\label{sec:lower-bound-via-capacity}
Finally, we prove Theorem~\ref{thm:mu-lb-implies-f-lb}. Our construction is quite similar to the construction in the previous section, except for a different choice of prime bucketing.

\begin{proof}[Proof of Theorem~\ref{thm:mu-lb-implies-f-lb}]
Let $\mu:=\mu_k^{\mathrm S}$. If $\mu\le 1$ the statement is trivial, so assume $\mu>1$ and fix $\varepsilon>0$.

\medskip\noindent
\textbf{1. A large cosunflower-free family on $[t]$.}
Recall that a $k$-sunflower is a collection of $k$ distinct sets whose pairwise intersections
are all equal. Dually, a $k$-cosunflower is a collection of $k$ distinct sets
whose pairwise unions are all equal. Taking complements in $[n]$ bijects $k$-sunflowers
and $k$-cosunflowers.

Choose $\lambda$ with $1<\lambda<\mu$ and
\begin{equation}\label{eq:lambda-close-to-mu}
\log\lambda>\log\mu-\frac{\varepsilon}{2}.
\end{equation}
By~\eqref{eq:lim}, there exists $t_0>0$ such that, for $t \geq t_0$ there always exists a $k$-cosunflower-free set family $\mathcal{F}_t \subset 2^{[t]}$ with $\abs{\mathcal{F}_t} \geq \lambda^t$.

\medskip\noindent
\textbf{2. Prime bucketing into $t$ disjoint blocks of harmonic sum $\ge 1$.}
Now fix $N$ large (how large will be specified later) and set
\[
L:=\log\log N,\qquad t:=\left\lfloor L-2\log L\right\rfloor,\qquad
\delta:=\frac{1}{L},\qquad y:=L^2,\qquad x:=N^{1/t}.
\]
For $N$ sufficiently large we have $t\ge t_0$, $y<x$, and for every prime $p>y$ we have $1/p\le 1/y\le\delta$. By Mertens' theorem,
\[
\sum_{p\le u}\frac1p=\log\log u+O(1)\qquad(u\to\infty),
\]
so
\begin{equation}\label{eq:enough-harmonic-mass}
\sum_{y<p\le x}\frac1p= \left(L- \log t\right) -\log\log y + O(1) \geq t(1+\delta)
\end{equation}for $N$ sufficiently large.

We now construct disjoint sets of primes $P_1,\dots,P_t\subseteq(y,x]$ greedily.
List the primes in $(y,x]$ in increasing order.
Start with $P_1=\varnothing$ and keep adding the next unused prime to the current $P_i$
until $\sum_{p\in P_i}\frac1p\ge 1$, then start $P_{i+1}$ and continue.
Since every added prime satisfies $1/p\le\delta$, each completed bucket satisfies
\begin{equation}\label{eq:bucket-sum}
1 \le \sum_{p\in P_i}\frac1p < 1+\delta.
\end{equation}
Thus \eqref{eq:enough-harmonic-mass} guarantees that at least $t$ buckets are completed. Fix such $P_1,\dots,P_t$ and set
\[
J_i:=\sum_{p\in P_i}\frac1p\in[1,1+\delta)\qquad(1\le i\le t).
\]

\medskip\noindent
\textbf{3. Encode $\mathcal F_t$ by squarefree integers $\le N$.}
For each $F\in\mathcal F_t$ define the set of squarefree integers
\[
\mathcal A(F):=\left\{ \prod_{i\in F}p_i: p_i\in P_i\ \text{for each }i\in F\right\},
\qquad
A:=\bigcup_{F\in\mathcal F_t}\mathcal A(F).
\]
Because the buckets $P_i$ are pairwise disjoint and we choose at most one prime from each,
each element of $A$ is squarefree and has all prime factors $\le x$.
Also $|F|\le t$, hence every $a\in A$ satisfies $a\le x^{|F|}\le x^t=N$, so $A\subseteq[N]$.

\begin{claim}\label{claim:a_is_k_good_1_2_3}
    $A$ is LCM-$k$-free.
\end{claim}
\begin{proof}[Proof of Claim~\ref{claim:a_is_k_good_1_2_3}]
Let $\phi$ map a natural number to its set of prime divisors. Then $\phi(A)$ is precisely the blow-up $B(\mathcal F_t, \{P_i\}_{i = 1}^t, \emptyset)$. Since all $a \in A$ are squarefree, $\phi(\lcm(a,b))=\phi(a)\cup\phi(b)$. By Proposition~\ref{prop:blowup-ok}, $\phi(A)$ is $k$-cosunflower-free, so $A$ is LCM-$k$-free.
\end{proof}

\medskip\noindent
\textbf{4. Compute the harmonic sum.}
For each $a\in A$ define its index set
\[
I(a):=\{i\in[t]: a \text{ is divisible by some prime in }P_i\}.
\]
Since the buckets are disjoint and each $a\in A$ contains at most one prime from each bucket,
we have $\mathcal A(F)=\{a\in A: I(a)=F\}$, and hence the families $\{\mathcal A(F)\}_{F\in\mathcal F_t}$
are pairwise disjoint. Therefore
\[
\sum_{a\in A}\frac1a=\sum_{F\in\mathcal F_t}\ \sum_{a\in\mathcal A(F)}\frac1a.
\]
For each $F\in\mathcal F_t$ we have
\[
\sum_{a\in \mathcal A(F)}\frac1a
=\prod_{i\in F}\Bigl(\sum_{p\in P_i}\frac1p\Bigr)
=\prod_{i\in F}J_i
\ \ge\ 1,
\]
using \eqref{eq:bucket-sum}. Hence
\[
\sum_{a\in A}\frac1a = \sum_{F\in\mathcal F_t}\sum_{a\in \mathcal A(F)}\frac1a \ge \sum_{F\in\mathcal F_t}1 = |\mathcal F_t|.
\]
Since $A$ is LCM-$k$-free, we obtain
\[
f_k(N) \ge \sum_{a\in A}\frac1a \ge |\mathcal F_t| \ge \lambda^{t}.
\]Recall that $t=L-2\log L+O(1)$ with $L=\log\log N$, hence $t=L-o(L)$.
Therefore
\[
\lambda^{t}
=\exp\left(t \log\lambda\right)
=(\log N)^{\log\lambda-o(1)}.
\]
By~\eqref{eq:lambda-close-to-mu}, $\log\lambda>\log\mu-\varepsilon/2$, so for $N$ sufficiently large
the $o(1)$ term is at most $\varepsilon/2$, and hence 
\[
f_k(N) \ge (\log N)^{\log\mu-\varepsilon}.
\]
This proves the theorem.
\end{proof}

\appendix

\section{Proof of Lemma~\ref{lem:harmonic-SS}}\label{appendix:SSS}

Let $L := \log\log N$. Fix $0<\eta<1$. For $\ell \ge 1$ and $x\ge 3$ let
\[
A_{\ell}(x)
:= \#\{ n\le x : \omega(n)=\Omega(n)=\ell \}
\]
be the counting function of squarefree $\ell$-almost primes. The squarefree version of the Sathe--Selberg theorem (see for instance~\cite[p.~237, ex.~4]{MoVa06} or~\cite[Eq.~(4.9)]{LJD20}) asserts that there exists a
continuous function $G(z)$ on $[0,2)$, given explicitly by
\[
G(z)
= \frac{1}{\Gamma(1+z)} \prod_{p}\Bigl(1 + \frac{z}{p}\Bigr)\Bigl(1-\frac1p\Bigr)^z,
\]
such that for every $\varepsilon\in(0,2)$, uniformly for integers $\ell$ with
$1\le \ell \le (2-\varepsilon)\log\log x$, one has
\begin{equation}\label{eq:sathe-selberg-squarefree}
A_{\ell}(x)
= G\left(\frac{\ell-1}{\log\log x}\right)
  \frac{x}{\log x}
  \frac{(\log\log x)^{\ell-1}}{(\ell-1)!}
  \left(1 + O_\varepsilon\left(\frac{\ell}{(\log\log x)^2}\right)\right).
\end{equation}
Note that $G(z)>0$ for all $z\in[0,2)$, since each Euler factor and
$1/\Gamma(1+z)$ is positive on $[0,2)$.

\medskip
Now fix $N\ge 3$ large and let $\ell$ be an integer with
$1\le \ell\le (1-\eta)L$.

\smallskip\noindent
\emph{The case $\ell=1$.} In this case
\[
H_1(N) = \sum_{\substack{n\le N\\ \omega(n)=\Omega(n)=1}}\frac1n
= \sum_{p\le N} \frac1p
= \log\log N + O(1) = L + O(1),
\]
by Mertens' theorem. Hence $H_1(N)\gg L = L^1/1!$, which is
\eqref{eq:H-ell-lower} for $\ell=1$.

\smallskip\noindent
\emph{The case $\ell\ge 2$.} Then $2\le \ell\le (1-\eta)L$.
By partial summation, we have
\[
H_\ell(N)
= \sum_{\substack{n\le N\\ \omega(n)=\Omega(n)=\ell}} \frac1n
= \frac{A_\ell(N)}{N} + \int_1^N \frac{A_\ell(t)}{t^2} dt.
\]
Since $A_\ell(t)/t\ge 0$ for $t \geq 1$, we have
\begin{equation}\label{eq:H-ell-int}
H_\ell(N) \ge \int_2^N \frac{A_\ell(t)}{t^2} dt.
\end{equation}

We next apply \eqref{eq:sathe-selberg-squarefree} with $x=t$. To ensure that the uniformity range for $\ell$ is satisfied, we restrict the
integral to a suitable subinterval where $\log\log t$ is comparable to $L$.
Set \(I := \left\{t\in[2,N] : \tfrac12 L \le \log\log t \le L \right\}\). Thus, for $t\in I$,
\[
\ell \le (1-\eta)L \le 2(1-\eta)\log\log t.
\]
Choose $\varepsilon=2\eta$, so that $2(1-\eta) = 2-\varepsilon$. Then for
all $t\in I$ and all $2\le \ell\le (1-\eta)L$ we have
\[
1\le \ell\le (2-\varepsilon)\log\log t.
\]
Hence \eqref{eq:sathe-selberg-squarefree} applies to $A_\ell(t)$ with $x=t$ for all $t\in I$, giving
\begin{equation}\label{eq:A-ell-asymp}
A_\ell(t)
= G\left(\frac{\ell-1}{\log\log t}\right)
  \frac{t}{\log t}
  \frac{(\log\log t)^{\ell-1}}{(\ell-1)!}
  \left(1 + O_{\eta}\left(\frac{\ell}{(\log\log t)^2}\right)\right)
\end{equation}
for all $t\in I$. Substituting \eqref{eq:A-ell-asymp} into \eqref{eq:H-ell-int} and restricting the integral to $t\in I$, we obtain
\begin{align*}
H_\ell(N)
 &\ge \int_{t\in I} \frac{A_\ell(t)}{t^2} dt\\
&= \frac1{(\ell-1)!}\int_{t\in I}
   G\left(\frac{\ell-1}{\log\log t}\right)
   (\log\log t)^{\ell-1}
   \frac{dt}{t\log t}
   \left(1 + O_{\eta}\left(\frac{\ell}{(\log\log t)^2}\right)\right).
\end{align*}For $t\in I$ we have $\log\log t\asymp L$, and since
$\ell\le (1-\eta)L$, we have
\[
\frac{\ell}{(\log\log t)^2}
\ll_\eta \frac1L.
\]
Thus, for $N$ sufficiently large, the factor
$1 + O_{\eta}\bigl(\ell/(\log\log t)^2\bigr)$ is bounded below by, say, $1/2$ on~$I$.
Moreover, for $t\in I$ we have
\[
0 \le \frac{\ell-1}{\log\log t}
\le \frac{(1-\eta)L}{L/2}
= 2(1-\eta) < 2,
\]
so that the argument of $G$ stays in the compact interval
$[0,2(1-\eta)]\subseteq[0,2)$. By continuity and positivity of $G$ on $[0,2)$,
there exists a constant $c_\eta>0$ such that
\[
G(z) \ge c_\eta \qquad\text{for all } 0\le z\le 2(1-\eta).
\]Therefore, using the substitution $u=\log\log t$ (so $du = dt/(t\log t)$), we
obtain
\begin{align*}
H_\ell(N)
&\ge \frac{1}{2(\ell-1)!} \int_{t\in I}
   G\left(\frac{\ell-1}{\log\log t}\right)
   (\log\log t)^{\ell-1}
   \frac{dt}{t\log t}
\\&\ge
\frac{c_\eta}{2(\ell-1)!}
\int_{\frac12 L}^{L} u^{\ell-1} du.
\end{align*}
The last integral can be evaluated explicitly:
\[
\int_{\frac12 L}^{L} u^{\ell-1} du
= \frac{L^\ell}{\ell}\left(1 - 2^{-\ell}\right)
\ge \frac{L^\ell}{2\ell}.
\]Hence
\[
H_\ell(N)
\ge
\frac{c_\eta}{2(\ell-1)!}\cdot \frac{L^\ell}{2\ell}
= \frac{c_\eta}{4} \frac{L^\ell}{\ell!}.
\]This shows that for all sufficiently large $N$ and all integers
$2\le \ell\le (1-\eta)L$ we have
\[
H_\ell(N) \ge C_\eta\,\frac{L^\ell}{\ell!}
\]
with $C_\eta = c_\eta/4>0$ depending only on~$\eta$. Together with the case
$\ell=1$ treated above, this proves that
\[
H_\ell(N) \gg_\eta \frac{L^\ell}{\ell!}
\]
uniformly for $1\le\ell\le(1-\eta)L$ as $N\to\infty$. This is~\eqref{eq:H-ell-lower}, and the lemma is proved.\qed

\end{document}